\documentclass[12pt]{article}
\usepackage{amsmath}
\usepackage{amsfonts}
\usepackage{graphicx}
\usepackage{color}
\usepackage{ntheorem}
\def\epsfig#1{}

\newtheorem{theorem}{Theorem}[section]
\newtheorem*{claim}{Claim}
\newtheorem{corollary}[theorem]{Corollary}
\newtheorem{lemma}[theorem]{Lemma}

\newtheorem{proposition}[theorem]{Proposition}
\newtheorem{remark}[theorem]{Remark}
\newenvironment{proof}[1][Proof]{\textbf{#1.} }{\ \rule{0.5em}{0.5em}}

\newcommand{\R}{\mathbf{R}}

\newcommand{\mean}{\overline}

\begin{document}

\author{
Jonathan Korman, Robert J. McCann, and Christian Seis\thanks{
\copyright 2014 by the authors. RJM is pleased to acknowledge the support of
Natural Sciences and Engineering Research Council of Canada Grants 217006-08.
This material is based in part upon work supported by the National Science Foundation under Grant No. 0932078 000, while the first two authors were in residence at the Mathematical Science Research Institute in Berkeley, California, during
the Fall of 2013.
\newline Department of Mathematics, University of Toronto, Toronto Ontario M5S 2E4 Canada. {\tt jkorman@math.toronto.edu, mccann@math.toronto.edu, cseis@math.toronto.edu}}  }

\title{Dual potentials for capacity constrained optimal transport}
\date{\today}

\maketitle

\begin{abstract}
Optimal transportation with capacity constraints, a variant of the well-known optimal transportation problem, is concerned with transporting one probability density $f \in L^1(\R^m)$ onto another one $g \in L^1(\R^n)$ so as to optimize a cost function $c \in L^1(\R^{m+n})$ 
while respecting the capacity constraints $0\le h \le \bar h\in L^\infty(\R^{m+n})$.

A linear programming duality theorem for this problem was first established by Levin.
% and rediscovered by the authors in see also \cite{KormanMcCannSeis13p}.  
In this note, we prove under mild assumptions on the given data, the existence of a pair of $L^1$-functions optimizing the dual problem. Using these functions, which can be viewed as Lagrange multipliers to the marginal constraints $f$ and $g$, we characterize the solution $h$ of the primal problem. We expect these potentials to play a key role in any further analysis of $h$.

Moreover, starting from Levin's duality,
we derive the classical Kantorovich duality for unconstrained optimal transport. In tandem with results obtained in our companion paper \cite{KormanMcCannSeis13p}, this amounts to a new and elementary proof of  
Kantorovich's duality. 

\end{abstract}

\section{Introduction}

In the optimal transport problem of Monge \cite{Monge81} and Kantorovich \cite{Kantorovich42}, 
one is given distributions
$f(x)$ % $f \in L^1(\R^m)$ 
of sources and $g(y)$ %$g \in L^1(\R^n)$ 
of sinks,  and is asked which pairing $h(x,y) \ge 0$ of sources with sinks
minimizes a given transportation cost $c(x,y)$.  The duality theory initiated by Kantorovich 
provides a key tool for the analysis of this question. It was historically one of the first and has become one  of the archetypal examples in the 
theory of infinite-dimensional linear programming. Note that the optimal transportation problem is concerned only with moving sources to sinks and not with the means by which this might be accomplished.

In the capacity-constrained variant
of this problem, a bound $\bar h(x,y)$ is imposed on the number of sources at $x$
which can be paired with sinks at $y$. Such a bound is meant to model the constraints on the means of transportation. For example, in a discrete setting, if bread loaves are transported from bakeries at $i$ to caf\'es at $j$ by means of trucks (each dedicated to route $ij$), then the capacity bound models each truck's capacity \cite{KormanMcCann12p}. From the perspective of the constrained transport problem we sometimes refer to the (unconstrained) optimal transport problem as the case $\bar h =\infty$.

The imposition of such a bound has a qualitative effect on the optimal solution and can lead to surprising conclusions about the topology of its support. For generic $c \in C^2$,
the unconstrained problem leads to solutions which are singular measures concentrated on sets of zero volume, whereas in the constrained problem all solutions are functions of the form $\bar h 1_W$ whose support $W$ is a Lebesgue set with positive volume. Examples show that the boundary of the support need neither be smooth nor simply connected
%not be smooth and that the interior can have holes 
\cite{KormanMcCann13}.

Although a duality theorem for capacity-constrained transport was established by 
Levin\footnote{In a private communication, Rachev and R\"uschendorf attribute Theorem
4.6.14 of \cite{RachevRuschendorf98} to a handwritten manuscript of Levin.
We do not know whether or not this manuscript was ever published.}
% which, so far as we know,has not been published.}
 (see Theorem 4.6.14 of 
\cite{RachevRuschendorf98}, for which a new proof is given in \cite{KormanMcCannSeis13p}),
it has not been clear until now whether the dual problem admits solutions.   
When they exist, such solutions play the role %akin to
of Lagrange multipliers for the primal problem. 
However, the compactification techniques used to find them in the unconstrained problem
\cite{McCannGuillen13} \cite{Villani09} fail miserably when $\bar h \ne +\infty$.
% when they exist.
%As in the unconstrained problem,
%the existence of such solutions is a key question which must be resolved if one hopes to characterize 
%primal solutions or investigate their analytic and geometric features.
The purpose of this manuscript is to prove that, under suitable hypotheses 
such as the continuity and strict positivity of $f,g$ and $\bar h$ on their %respective
supports, presumed compact in $\R^m, \R^n$ and $\R^{m+n}$, Levin's dual problem 
in fact admits solutions $(u,v)\in L^1 \oplus L^1$.  
This allows us to characterize %\eqref{complementary}-\eqref{slackness} 
optimizers $h$ for the primal problem in terms of Lagrange multipliers $(u,v)$.
%using complementary slackness.  
As in the theory for the unconstrained %variant of the 
Monge-Kantorovich problem %\cite{Brenier91} 
%surveyed in 
\cite{McCannGuillen13} \cite{Villani09} 
which has developed since the work of Brenier \cite{Brenier87},
we expect our characterization of primal optimizers using dual solutions
will be the %fundamental 
starting point for any future analysis of their analytic or geometric properties.

As a consequence of our analysis, we are also able to deduce Kantorovich's duality theorem for the unconstrained optimal transportation problem as a singular limit $\bar h\to \infty$.
%conclusions:  the extreme points of the feasible set take the form $h = 1_W \bar h$ for some
%unknown set $W\subset \R^m \times \R^n$,  where

To fix notation, let $f$ and $g$ be probability densities 
on $X\subset \R^m$ and $Y\subset \R^n$, respectively.
We assume $X$ and $Y$ to be compact, unit volume sets without loss of generality.
%Let $B(X)$ denote the set of bounded Borel functions on $X$,
%and $M(X)$ the set of Borel measures with finite total variation.
The continuous functions $C(X)$ form a Banach space under the
supremum norm $\|\,\cdot\,\|_\infty$, whose dual space consists
of signed Radon measures $M(X)=C(X)^*$ normed
by total variation $\|\,\cdot\,\|_{TV}$,  according to the Riesz-Markov theorem \cite{ReedSimon80}.
Let $C_+(X) = \{c \in C(X) \mid c \ge 0\}$ %, $C_+(X) = C(X) \cap B_+(X)$,
and  $M_+(X) =  \{ U \in M(X) \mid U \ge 0\}$. 
% and $C_+(X) = C(X) \cap B_+(X)$.  
We denote the Jordan 
decomposition of a signed measure into its positive and negative parts by  $U=U_+ - U_-$,
and %its %absolute value magnitude by 
set $|U| = U_+ + U_-$. % to be the total variation of $U$.

Fix $ 0 \le \bar h \in L^\infty(X\times Y)$ 
%whose marginals dominate the probability densities  $f \in L^\infty(X)$ and $g \in L^\infty(Y)$.  
and let $\Gamma:=\Gamma^{\bar h}(f,g)$ denote the set of joint probability densities 
$h\le \bar h$ with $f$ and $g$ as marginals. Kellerer \eqref{iffKellerer}
%\cite{Kellerer64m} \cite{Kellerer64f} 
and Levin \eqref{iffLevin} %\cite{Levin84}
give necessary and sufficient conditions for $\Gamma$ to be non-empty: $\Gamma\not=\emptyset$ if and only if
\begin{equation}\label{iffKellerer}
\int \varphi f dx + \int \psi g dy \le \iint \left[\varphi(x) +\psi(y)\right]_+ \bar h(x,y) dxdy
%\quad\mbox{for all }\varphi,\psi\in L^{\infty}
\end{equation}
for all $\varphi,\psi \in L^\infty$ \cite{Kellerer64m} \cite{Kellerer64f},  or equivalently if and only if 
\begin{equation}\label{iffLevin}
\int_A f dx + \int_B g dy \le 1 + \iint_{A \times B} \bar h  dx dy
\end{equation}
for all Borel $A \times B \subset X \times Y$ \cite{Levin84}.
Here $dx = dH^m(x)$ and $dy=dH^n(y)$ denote Lebesgue measure. 
Assuming these conditions
are satisfied,  the capacity-constrained transportation problem is to compute
\begin{equation}\label{primal}
I^* := \sup_{h \in \Gamma^{\bar h}(f,g)} \iint h(x,y) s(x,y) dx dy,
\end{equation}
where our sign convention is to maximize the surplus $s:=-c \in L^1(X \times Y)$ rather than to minimize the cost $c$.
This supremum is known to be attained at an extreme point of the feasible set \cite{KormanMcCann12p},
and the extreme points have been characterized in \cite{KormanMcCann13} as those $h \in \Gamma$ given by $h = \bar h 1_W$ for some Lebesgue measurable $W \subset X \times Y$.

For $u \in L^1(X)$, $v \in L^1(Y)$ and $s \in L^1(X \times Y)$ define
\begin{equation}\label{I}
I(u,v) %&:=& I^{(s,\bar h)}(u,v; f, g) \\ 
:= \iint[s(x,y) + u(x) + v(y)]_+ \bar h(x,y) dxdy
- \int u f dx - \int v g dy.
\end{equation}
%integrals are always with respect to Lebesgue measure unless otherwise indicated.
Although Levin's dual problem is linear rather than convex and was set in $C(X) \oplus C(Y)$ rather than in $L^1(X) \oplus L^1(Y)$, it is easily seen to be 
equivalent to computing the infimum %We would like to know whether
\begin{equation}\label{convex dual}
I_* := \inf_{(u,v) \in L^1(X) \oplus L^1(Y)} I(u,v),
\end{equation}
%over continuous functions $u$ and $v$, 
by %approximation and
identifying his Lagrange multiplier $w$ conjugate to the capacity constraint 
$h(x,y) \le \bar h(x,y)$
with $w(x,y) = [s(x,y) + u(x) + v(y)]_+$.
Levin %'s result, reproved in \cite{KormanMcCannSeis13p}, 
asserts
\begin{equation}\label{levin}
I^*=I_*,
\end{equation}
see also \cite{KormanMcCannSeis13p}.
We would like to know whether the infimum $I_*$ is attained. % for some $p \in [1,\infty]$.

Our main theorem is stated below. It gives a condition on continuous strictly positive $\bar h,f,g$ which ensures that the infimum in \eqref{convex dual} 
is attained: namely, 
 the set $\Gamma^{\bar h/\eta}( f, g)$ of joint probability densities $h\le \bar h/\eta$
with $f$ and $g$ as marginals must be non-empty --- not only for $\eta=1$ as required to have the identity
$ I_* = I^*$  %from \eqref{eq: weak duality} below
 --- but for some $\eta>1$.
%This leads to the following characterization
%of optimizers $h$ for the primal problem \eqref{primal}.

\begin{main*}[Theorem \ref{T:existence of optimizers}]
Let $f,g$ and $\bar h$ be continuous and strictly positive on the compact, unit-volume sets 
$X \subset \R^m,Y\subset \R^n$ and $X \times Y$ respectively. % of unit volume. 
Fix $\eta>1$ and $s \in L^1(X \times Y)$. If $\Gamma^{\bar h/\eta}(f, g)$ is not empty,
%If there exists a probability density $h \le \bar h/\eta$ having $f$ and $g$ as its marginals,
then there exist functions $(u,v) \in L^1(X) \oplus L^1(Y)$ such that $I_*=I(u,v)$.
\end{main*}

The functions $u$ and $v$ are commonly referred to as {\em dual potentials}.
%From the perspective of the unconstrained transportation problem, the significance of this result 
Combining this theorem with known results 
(Proposition \ref{P:duality} and Levin's duality \eqref{levin}), 
we obtain a %complementary slackness 
characterization of optimality:

\begin{corollary}[Characterization of optimality]\label{C:Characterizing optimality}
Under the hypotheses %and notation 
of the main Theorem, any 
$h \in \Gamma^{\bar h}(f,g)$ is optimal if and only if there exist 
$(u,v) \in L^1(X) \oplus L^1(Y)$ such that
\begin{equation}\label{complementary slackness1}\displaystyle
s(x,y) + u(x) + v(y)\;\left\{\begin{array}{lll}\le\; 0&\quad \mbox{      } 
{\rm where}\  h=0,\\
=\; 0&\quad \mbox{      }{\rm where}\ 0<h< \bar h,\\
\ge\; 0&\quad\mbox{      } {\rm where}\ h= \bar h . \end{array} \right.
\end{equation}
 %\ref{T:existence of optimizers}, 
%the infimum \eqref{convex dual} is attained
%at some point $ I(u,v) = I_*$ and conditions \eqref{complementary slackness}
%become necessary and sufficient to characterize those $h$ which maximize \eqref{primal}.
\end{corollary}

Characterization \eqref{complementary slackness1} can be interpreted as a 
complementary slackness condition.

We now outline our strategy.
Proving that the minimum \eqref{convex dual} is attained
relies on a continuity and compactness argument. At its core is a coercivity estimate (Proposition \ref{P:coercivity}) which guarantees $L^1$-boundedness of (suitably normalized) minimizing sequences. In view of the lack of compactness of the closed unit ball in the non-reflexive Banach space $L^1$, it is natural to embed $L^1$ into the larger space of signed Radon measures equipped with the total variation norm, which, as a dual space, is guaranteed better compactness properties by Alaoglu's theorem. Although the number of competitors increases, this embedding does not affect our variational problem: It can be easily shown that the value of the functional $I(u,v)$, suitably extended to singular measures, does not decrease with this extension. Lower semicontinuity properties of $I(u,v)$ allow us to conclude the proof.

The $L^1$-bounds that we derive for minimizing sequences are in a certain sense uniform in $\bar h$. This observation allows us to derive the classical Kantorovich duality from Levin's duality as the singular limit $\bar h\to \infty$. In tandem with the results obtained in the companion paper \cite{KormanMcCannSeis13p}, this amounts to a new and {\em elementary} proof of the Kantorovich duality theorem.

The remainder of this paper is organized as follows: In Section \ref{S2}, we derive weak duality and complementary slackness conditions; Section \ref{S3} contains the key (coercivity) estimates; Section \ref{S4} contains the main result; 
in Section \ref{S5}, we give a new elementary proof of the Kantorovich duality; we finally conclude this paper with a discussion on future work in Section \ref{S6}.\\

\noindent {\bf Acknowledgement}. We would like to acknowledge Robert Jerrard for fruitful discussions.

\section{Complementary slackness}\label{S2}

In the following proposition, we establish the complementary slackness conditions of linear programming and obtain the inequality $I_* \ge I^*$, that is, the easy part of Levin's duality, as a by-product.

\begin{proposition}[Complementary slackness]\label{P:duality}
Fix $(u,v) \in L^1(X) \oplus L^1(Y)$, $(s,\bar h) \in (L^1\oplus L^\infty)(X \times Y)$ 
and a probability density $0 \le h \le \bar h$ whose marginals are denoted by $f$ and $g$.  
%with $\bar h\ge 0$and probability densities $f$ on $\R^m$ and $g$ on $\R^n$.
%Let $S\in M(X \times Y)$ denote the measure with Lebesgue density $s$.
Then 
\begin{equation}\label{easy inequality}
 I(u,v) \ge \iint h s dxdy.
\end{equation}
Moreover, equality holds if and only if $(u,v)$ and $h$ satisfy the following complementary slackness conditions
\begin{equation}\label{complementary slackness}\displaystyle
s + u + v\;\left\{\begin{array}{lll}\le\; 0&\quad \mbox{   in   } \{ h=0 \},\\
=\; 0&\quad \mbox{    in  } \{ 0<h< \bar h \},\\
\ge\; 0&\quad\mbox{  in    } \{ h= \bar h \} ,\end{array} \right.
\end{equation}
up to Lebesgue negligible sets.
\end{proposition}

Clearly, $(u,v)\in L^1(X)\oplus L^1(Y)$ and $h\in L^{\infty}(X\times Y)$ satisfying \eqref{complementary slackness} are a pair of minimizers and maximizers of \eqref{convex dual} and \eqref{primal}, respectively.

\begin{proof} Because $h\le \bar h$, we first notice that
\begin{eqnarray*}
I(u,v) &\ge& \iint [s+u+v]_+  h dxdy- \int ufdx -\int vgdy\\
&\ge & \iint \left(s+u+v\right) h dxdy- \int ufdx -\int vgdy\\
&=&  \iint h s dxdy
\end{eqnarray*}
to establish \eqref{easy inequality}.
Moreover, these inequalities become equalities if and only if $s + u+ v\le 0$ where $\bar h> h$ and $s+u+v\ge 0$ where $h > 0$. Combining these two conditions implies \eqref{complementary slackness}. 
%To deduce the first statement of the proposition, it remains to invoke the marginal constraint on $h$.
\end{proof}

\section{Coercivity in $L^1$}\label{S3}

In this section. we establish  the estimates which constitute the core of this paper.
%It culminates in the proof that optimizers for \eqref{relaxed dual} exist which are measures  $(U,V) \in M(X) \oplus M(Y)$.

%Fix compact sets $X \subset \R^m$ and $Y \subset \R^n$ with unit volume, and
%functions $(s,\bar h) \in (L^1 \oplus C)(X\times Y)$,
%with the left and right marginals of $\bar h\ge 0$ dominating the probability densities $f \in C(X)$ and $g \in C(Y)$.
%
We remark that $I(u+k,v-k) = I(u,v)$ for all $k\in \R$ shows the level sets of $I$ cannot be compact.
To resolve this lack of coercivity,  we shall eventually assume
%\begin{equation}\label{equal means}
\[
\mean {uf} := \int u(x)f(x) dx = \int v(y)g(y)dy =: \mean{vg},
\]
%\end{equation}
which can always be enforced by an appropriate choice of constant $k$,
and restricts the problem to a codimension 1 subspace of $L^1(X) \oplus L^1(Y)$.
We would then like to show $I(u,v) \le I_*+1$ implies an $L^1$-bound on $u$ and $v$ under suitable assumptions on the data $f$, $g$, and $\bar h $. . 
%A minimizing sequence will then have a weak-$*$ limit in $M_+(X\times Y)$.

%\marginpar{(Ideally, by some generalization of the convexity argument 2.13 of Lieb and Loss, 
 %it would have a strong limit in $L^1$.)}

%To  bound $u$ and $v$ in $L^1$, we  make the simplifying assumptions
%that $f$ and $g$ are strictly positive and strictly dominated by the marginals of $\bar h$.
%%Due to Lemma \ref{L:polynomial approximation}, it costs no generality to work with functions $(u,v)$ as opposed to measures.
%%assume $f,g,u$ and $v$ are polynomials.
%Under such assumptions, we
%shall first show that $I(u,v) \le I_*+1$ implies a bound on the means $\mean{uf}=\mean{vg}$.
Our argument comes in two steps: We first prove a bound on the means $\mean{uf} = \mean{vg}$. 
We shall %then show the high ranking percentiles of $uf$ are bounded below,  and 
use  this to establish a bound on the oscillation $\|uf - \mean{uf}\|_{L^1}$ of $uf$.
A similar bound for $vg$ follows by symmetry.

%Set $\mean s = \int_{X \times Y} s$. 
The bound on the means follows immediately from $I(u,v) \le I_*+1$ 
by the following claim.

%\marginpar{forces $I(u) >0$ if $\mean s \ge 0$}
\begin{lemma}[Mean bound]
\label{L:mean bound}
%If $f=g=1<\bar h = const$ then
%$I(u) \le I_* + 1$ implies
%\le \mean u \le \frac{1}{2}I(u) 
Fix $(u,v,s) \in L^1(X) \oplus L^1(Y) \oplus L^1(X \times Y)$ and a probability density $h\in L^\infty(X \times Y)$
with marginals $f$ and $g$.  If $\bar h \ge \eta h$ for some $\eta>1$, 
%$\bar h \in C_+(X \times Y)$ such that the left and right marginals
%of $\bar h/\eta$ dominate probability densities $f \in B(X)$ and $g \in B(Y)$.
%If $\Gamma^{\bar h/\eta}(f,f)$ is non-empty for some $\eta>1$ 
%$\bar h(x,y) \ge \epsilon f(x)f(y)$ for some $\epsilon>0$ and all $x,y \in X$ 
then
%\begin{equation}\label{mean bound}
\[
- I(u,v) \le  %\epsilon 
\mean{uf} + \mean{vg} \le \frac{I(u,v) + \| \eta h  s\|_{L^1}}{\eta-1}.
\]
%\end{equation}
\end{lemma}

\begin{proof}
One direction is easy: $I(u,v) \ge -\mean{uf} - \mean{vg}$ follows directly from the definition \eqref{I}.
%Letting $h \in \Gamma^{\bar h/\eta }(f,f)$, 
The other direction is a consequence of %Jensen's inequality:
\begin{eqnarray*}
I(u,v) &\ge& - \mean{uf} - \mean{vg} + \eta \iint h(x,y)[s(x,y) + u(x) + v(y) ]dxdy 
%\\ &=& \eta \overline {h c} +  2(\eta-1) U(f)  
\\&\ge&(\eta-1)(\mean{uf}+\mean{vg})  - \| \eta h s\|_ {L^1} 
\end{eqnarray*}
since $h$ has $f$ and $g$ as its left and right marginals.
\end{proof}

%Let $u \in M(X)$ be the measure given by $du =fdU$, and note that the total
%mass $\mean u=U(f)$ of $u$ can also be viewed as its Lebesgue mean over the unit 
%volume set $X$.
We next 
%bound the high ranking percentiles
%of $uf$ from below in case $I(u,v)\le I_* + 1$.  
%The argument is more transparent in the special case that $\mean {uf}, \mean{vg}$ and $s$ vanish.
%%Set  $2B := \max\{I_*+1, \frac{I_*+1-\bar h \mean c}{\bar h-1}\}$ so that 
%%the preceding claim implies $|\mean u| \le B$. 
%Define $\uf_{p} \in \R$ so that 
%\begin{equation}\label{top third}
%|\{x \mid u (x) f(x) \ge \uf_{p} \}| \ge p %\frac13 |X| 
%\quad{\rm and}\quad
%|\{x \mid u(x) f(x) \le \uf_{p} \}| \ge 1-p %\frac23 |X |
%.\end{equation}
%In other words,  $(uf)^\#(p) = \uf_p$ is the equimeasurable non-increasing rearrangement
%$(uf)^\#:(0,1) \longrightarrow \R$ of $uf$.  For $p$ sufficiently small, the following claim
convert upper bounds on $f,g$ and $1/\bar h$ into a bound on 
the oscillation of $uf$ around its mean.
%$uf_p$, %and the 
Of course, a symmetrical bound holds for $vg$.

%(whose proof can also be adapted to yield a one-sided bound for $\uf_{\frac12+\epsilon}$):

\begin{lemma}[Oscillation bound]\label{L:oscillation bound}
Fix %compact sets $X \subset \R^m$ and $Y \subset \R^n$ with unit volume, and continuous  functions 
$u,f \in L^1(X)$, $v,g \in L^1(Y)$ and $(s,\bar h) \in (L^1 \oplus L^\infty)(X\times Y)$.
%with the left and right marginals of $\bar h$ dominating the probability densities $f $ and $g$.
%\label{quartile claim}
If %$u,v,f$ and $g$ are polynomials with
$\bar h(x,y) \ge \epsilon f(x)g(y) $
%and $\min\{f(x)^{\pm1}, g(y)^{\pm 1}\} \ge \epsilon$ 
for some $0<\epsilon\le 1$ and 
%some positive $\epsilon>0$  and 
all $(x,y) \in X \times Y$, then 
\begin{equation}\label{BMO}
\frac{\epsilon}{6} \|uf - \mean {u f}\|_{L^1(X)} \le I(u,v) + \|sfg\|_{L^1}+ |\mean{uf}| + |\mean{vg}|. 
\end{equation}
\end{lemma}

\begin{proof}
%By approximation of $U$ if necessary,  we may take $u=fU$ to be $L^1$.
Let $\sigma = \int |uf-\mean{uf}|dx$ denote oscillation around the mean,
i.e. the $L^1$-distance separating  $uf$ from its average value.
Set $X^\pm = \{x:\: \pm (u(x)f(x)-\mean {uf}) > \sigma/3\}$ and $x^\pm = |X^\pm|$. 
From the definition of $\mean{uf}$ we find
%\begin{eqnarray*}
%$$0 =[(1-x_+ - x_-) \mean \int_{|u - \mean u| \le \sigma/3} 
%+ x_+ \mean \int_{u -\mean u \ge \sigma /3} + x_-\mean \int_{u -\mean u \le -\sigma/3}](u-\mean u)f$$ whence
$$
%(1-x^+-x^-) -\!\!\!\!\!\!\!
\int_{\{ |uf -\mean{uf}| \le \sigma/3\}}(\mean{uf}- u(x)f(x))dx  = A^+ - A^- 
$$
where
\begin{eqnarray*}
A^\pm &=& \pm \int_{X^\pm} (u(x)f(x)-\mean{uf}) dx
\\ &\ge& 0.
\end{eqnarray*}
On the other hand, from the definition of $\sigma$,
\begin{eqnarray*}
\sigma &=&
%(1-x_+ - x_-) \mean 
(\int_{\{|uf - \mean{uf}| \le \sigma/3\}} 
+ \int_{X^+} + \int_{X^-}) |u(x)f(x)-\mean{uf}| dx
\\ & \le & (1-x^+-x^-) \sigma/3 + A^+ + A^-.
\end{eqnarray*}
Thus 
$$A^+ + A^- \ge (\frac23 + x^+ + x^-)\sigma \ge \frac23 \sigma 
$$
and 
$$A^+ - A^- \ge -(1-x^+ - x^-) \frac\sigma 3 \ge -\frac\sigma 3
$$
whence
$A^+ %:= \int_{\{x \mid u-\mean u \ge \sigma/3\}} |u-\mean u|f  
\ge \sigma/6. $  
Thus to control $\sigma$ it is enough to control $A^+$.

%Now,  a function $u$ can be far from $0$ in $L^1$ either because $\mean u$ is 
%large,  or because $\sigma$ is large,  or both.  Let's show that in our case 
%if $\mean u$ is bounded,  then $\sigma$ must also be bounded. 
%Among sets of Lebesgue measure $q$, % carrying one third of the mass of $f$,  
%fix $Y_{q}=\{ y \mid v(y)g(y) \ge \vg_{q} \}$ to be the one on which $vg$ takes its largest values,
%so that $\vg_{q}$ is defined analogously to \eqref{top third}.
%If $\mean u + \sigma/3 + \uf_{1/3} \ge \|c\|_\infty$ 
%Then
Now, since $g$ is a probability density
 \begin{eqnarray*}
A^+&=& - x^+ \mean{uf} + \iint_{X^+\times Y} ufg\, dx dy \\
&=& \iint_{X^+\times Y} (s+u+v)fg\, dx dy - \iint_{X^+\times Y} sfg\, dx dy\\
&&\mbox{} - x^+ \mean{uf} - \mean{vg} \int_{X^+}f\, dx \\
&\le& \frac1{\epsilon} I(u,v) +  \|sfg\|_{L^1} + (\frac1\epsilon-x^+)\mean{uf} + \left(\frac1\epsilon-\int_{X_+}f\, dx\right)\mean{vg}
\end{eqnarray*}
where in the last estimate we have used the assumption that $\epsilon fg\le \bar h$ for some $0<\epsilon\le1$. 
Using that $f$ is a probability density, $x^+\le1$, and $\epsilon\le 1$ yields
\[
\epsilon A^+ \le I(u,v) + \|sfg\|_{L^1} + \Big[\mean{uf}\Big]_+ + \Big[\mean{vg}\Big]_+.
\]
Since $A^+ \ge \sigma/6$, this yields the desired bound \eqref{BMO}, in which we have replaced
the positive part $[\cdot]_+$ by absolute value $|\cdot|$ simply for ease of parsing.
\end{proof}

%\begin{claim}[Mean bound]
%Fix $(u,v,s) \in L^1(X) \oplus L^1(Y) \oplus L^1(X \times Y)$ and a probability density $h\in B_+(X \times Y)$
%with marginals $f$ and $g$.  If $\bar h \ge \eta h$ for $\eta>1$, then
%\begin{equation}\label{mean bound}- I(u,v) \le  %\epsilon 
%\mean{uf} + \mean{vg} \le \frac{I(u,v) + \| \eta h  s\|_1}{\eta-1}.
%\end{equation}
%\end{claim}

%\begin{claim}[Percentile bound]\label{C:percentile bound}
%Fix %compact sets $X \subset \R^m$ and $Y \subset \R^n$ with unit volume, and continuous  functions 
%$u,f \in L^1(X)$, $v,g \in L^1(Y)$ and $s,\bar h \in L^\infty(X\times Y)$,
%with the left and right marginals of $\bar h$ dominating the probability densities $f $ and $g$.
%\label{quartile claim}
%If $u,v,f$ and $g$ are polynomials with
%$\bar h(x,y) \ge \epsilon f(x)g(y) $
%and $f(x) \in [\epsilon,1/\epsilon]$ and $g(y) \in [\epsilon,1/\epsilon]$
%%$\epsilon f_0 g_0>0$ %and $f(x) > f_0$ 
%for %some positive $\epsilon>0$  and 
%all $(x,y) \in X \times Y$, then $0<p<\epsilon^2$ implies
%%$\mean{uf} = \mean{vg}$ implies
%\begin{equation}\label{percentile bound}
%[uf_p]_- \le \frac1{(\epsilon^2-p)p}
%\left[I(u,v) + \epsilon\|s\|_\infty+(1-\epsilon^2p)(\mean{uf}+\mean{vg}) + \epsilon^2p[\mean{vg}]_+ \right].
%\end{equation}
%\end{claim}

\begin{proposition}[Coercivity]\label{P:coercivity}
%Fix $L^1$ functions $u,v$ and $s$ on $X,Y$ and $X \times Y$ respectively.
Let $s \in L^1(X\times Y)$ and $\eta h \le \bar h \in L^\infty(X \times Y)$, where $\eta>1$ and $h$ is a probability
density having $f$ and $g$ as its marginals. Assume $\bar h(x,y) \ge \epsilon f(x) g(y)$ 
%with $f(x) \in [\epsilon,1/\epsilon]$ and $g(y) \in [\epsilon,1/\epsilon]$ 
and $\min\{f(x),g(y)\}\ge \epsilon$
for some $\epsilon>0$ and almost all $(x,y) \in X \times Y$.
If $\mean{uf} = \mean{vg}$ for some $u\in L^1(X)$ and $v\in L^1(Y)$, then
$I(u,v) \le I_* +1$ implies $\|u\|_{L^1}$ and $\|v\|_{L^1}$ are controlled
by a bound which depends only on $I_*$, $\epsilon$, $\eta$, 
$\|s\|_{L^1}$ and $\|\bar h\|_\infty$.
\end{proposition}

\begin{proof}
%Fix $p=q=\epsilon^2/2$. % and notice $I_*$ depends only on $s$ and $\|\bar h\|_\infty$.  
According to  Lemma \ref{L:mean bound},  the means $\mean{vg}=\mean{uf}$ 
are bounded in terms of the given $I_*$, $\epsilon$, $\eta$,  $\|s\|_{L^1}$ and 
$\|\bar h\|_\infty$.
Lemma \ref{L:oscillation bound} and $\epsilon fg\le \bar h$ then bounds the mean oscillation of $uf$ around $\mean{uf}$ in terms
of $\mean{uf}=\mean{vg}$ and the listed parameters.  
Now $\|u\|_{L^1} \le \|f^{-1}\|_{L^\infty} \|uf\|_{L^1} \le \epsilon^{-1}\|uf\|_{L^1}$
combined with $\|uf\|_{L^1} \le \|\mean{uf}\|_{L^1} + \|uf-\mean{uf}\|_{L^1}$ yields the desired bound
for $\|u\|_{L^1}$.  A similar bound for $\|v\|_{L^1}$ follows by symmetry.
\end{proof}

\begin{remark}\label{R:uniform bound} The bound mentioned in the conclusion of Proposition \ref{P:coercivity} is uniform in $\|\bar h\|_{L^{\infty}}\gg1$.\end{remark}

\section{Existence of optimizers}\label{S4}
Proposition~\ref{P:coercivity} implies that we can chose a minimizing sequence $(u_i,v_i)$ for \eqref{convex dual} which is uniformly bounded in $L^1(X)\oplus L^1(Y)$. Since the $L^1$ unit ball is non-compact,  an $L^1$ bound is not enough to permit extraction of a limit lying in
$L^1(X)\oplus L^1(Y)$. To recover convergence we would like to view $(u_i,v_i)$ in a larger space which is better behaved.

Therefore, let us extend the definition of $I(u,v)$ 
from $L^1$ %(X) \oplus L^1(Y)$ 
to the space of signed measures 
%$M(X) \oplus M(Y)$ 
with finite total variation which, as a dual Banach space, has better
compactness properties.
%restrict our attention to the case in which $H$ is given
%by a bounded density $\bar h \in B_+(X \times Y)$ with respect
%to Lebesgue measure.  In the case the same is true for $F$ and $G$,
%and we denote their Lebesgue densities by $f$ and $g$ respectively.
Denoting by 
%Letting $H^m$ denote 
%$H^m$ 
%the Lebesgue measure on $\R^m$, 
%and $S$ denote the absolutely continuous measure on $X \times Y$ with Lebesgue density $s(x,y)$,
%and
by $S \in M(X \times Y)$ the measure with Lebesgue density $s(x,y)$,
this extension is given by
\[
\tilde I(U,V) 
:=\iint \bar h d[S + U \otimes H^n + H^m \otimes V]_+ - \int f dU - \int g dV 
\]
for $(U,V) \in M(X) \oplus M(Y)$. (We recall that in our notation, $H^m$ and $H^n$ are Lebesgue measures.) 
The next lemma is routine; it verifies that $\tilde I$ is lower semicontinuous with respect to weak-$*$ convergence in $M(X)\oplus M(Y)$.

\begin{lemma}[Lower semicontinuity]\label{L:lower semicontinuity}
%Given %a measure 
Let $s \in L^1(X\times Y)$ induce $dS = s dH^{m+n}$. Given continuous non-negative 
$f,g$ and $\bar h$ on the compact sets $X\subset \R^m, Y\subset \R^n$ and $X \times Y$, 
the functional $\tilde I(U,V)$
behaves lower semicontinuously with respect to weak-$*$ convergence in $M(X)\oplus M(Y)$.
%on norm-bounded subsets of 
%$M(X)\oplus M(Y)$, when %with respect to the
%topologized by weak-$*$ convergence of $(U,V)$.
% and uniform convergence of $(f,g)$.
\end{lemma}

\begin{proof}
Choose a bounded sequence in $M(X) \oplus M(Y)$ which converges
$(U_i,V_i) \to (U,V)$ when tested against functions in $C(X) \oplus C(Y)$.
 %Let continuous functions $(f_i,g_i) \to (f,g)$ uniformly.  
Then
\[
\int f dU = \lim_{i \to \infty} \int f dU_i \quad {\rm and} \quad
\int g dV = \lim_{i \to \infty} \int g dV_i, 
\]
so we need only show 
\begin{eqnarray}\label{lower semicontinuity}
\lefteqn{\liminf_{i \to \infty}  \iint \bar h d[S + U_i \otimes H^n + H^m \otimes V_i]_+ }\nonumber\\
 &&\mbox{}\ge \iint \bar h d [S + U\otimes H^n + H^m \otimes V]_+.
 \end{eqnarray}
We define the signed Radon measures $d\mu := \bar h d\left( S + U \otimes H^n + H^m \otimes V\right)$ 
and $d\mu^i := \bar h d\left(S + U_i \otimes H^n + H^m \otimes V_i\right)$. 
Then $\mu^i\to \mu$ weakly-$*$. Indeed, for every $\varphi\in C(X\times Y)$
\begin{eqnarray*}
\iint\varphi d\mu^i &=& \int \bar h s\varphi dxdy + \int \left(\int \bar h\varphi dy\right)d U^i + \int\left(\int \bar h\varphi dx\right) dV^i\\
&\to& \int \bar h s\varphi dxdy + \int \left(\int \bar h\varphi dy\right)d U + \int\left(\int \bar h\varphi dx\right) dV\\
&=&\iint\varphi d\mu,
\end{eqnarray*}
because the $x$- and $y$-marginals of $\bar h\varphi$ are both continuous. Moreover, decomposing $\mu$ and $\mu^i$ into their positive and negative parts $\mu = \mu_+ - \mu_-$ and
$\mu^i=\mu^i_+ - \mu^i_-$, \eqref{lower semicontinuity} becomes
\begin{equation}
\label{lsc}
\liminf_{i\to\infty} \mu^i_+(X\times Y)\ge \mu_+(X\times Y).
\end{equation}
This in fact is an immediate consequence of the weak-$*$ convergence $\mu^i\to \mu$: On the one hand, choosing $\varphi\equiv1$, weak-$*$ convergence implies that $\mu^i_+(X\times Y) - \mu_-^i(X\times Y)=\mu^i(X\times Y)\to \mu(X\times Y)=\mu_+(X\times Y) - \mu_-(X\times Y)$. On the other hand, by the lower semicontinuity of the total variation norm, it is $\mu_+(X\times Y) + \mu_-(X\times Y)=|\mu|(X\times Y)=\|\mu\|_{TV}\le\liminf_{i\to\infty} \|\mu^i\|_{TV}=\liminf_{i\to\infty}|\mu^i|(X\times Y)=\liminf_{i\to\infty}\left(\mu^i_+(X\times Y) + \mu^i_-(X\times Y)\right)$. The statement in \eqref{lsc} now follows upon combining both convergence results.
%
% Decompose $\mu := S + U \otimes H^n + H^m \otimes V$ and $\mu^i := S + U_i \otimes H^n + H^m \otimes V_i$ 
% into their positive and negative parts $\mu = \mu_+ - \mu_-$ and
%$\mu^i=\mu^i_+ - \mu^i_-$.  Our hypotheses imply
%$|\mu^i|(X\times Y) = \mu^i_+(X \times Y) + \mu^i_-(X \times Y)$
%is bounded and 
%$\mu^i \to \mu$ against functions in $C(X \times Y)$.
%By the Banach-Alaoglu theorem,  any subsequence of $\mu^i_\pm$ which minimizes
%\eqref{lower semicontinuity} has a further subsequence
%that converges to limits $\mu^\infty_\pm$.  Now
%$\mu^\infty_+ - \mu^\infty_- = \mu = \mu_+ - \mu_-$ so 
%$\mu^\infty_\pm = \mu_\pm + \nu$ where $\nu \ge 0$.
%Since $0 \le \bar h \in C(X \times Y)$, for this sub-subsequence we have 
%\begin{eqnarray*}
%\lim_{k\to\infty} \int_{X \times Y} \bar h d \mu^{i(j(k))}_+ 
%& =& \int_{X \times Y} \bar h d(\mu_+ + \nu)
%\cr &\ge& \int_{X \times Y} \bar h d\mu_+,
%\end{eqnarray*}
%the desired conclusion.
 \end{proof}

To prove the next theorem we need to recall some basic facts about signed Borel measures. 
%The Lemma below is used in the proof of  Theorem~\ref{T:existence of optimizers}.

A measure $\mu$ on $\R^m$ is absolutely continuous with respect to Lebesgue measure $H^m$, denoted $\mu \ll H^m$, if for any $H^m$-measurable set $\Omega$ for which $H^m (\Omega)=0$, it is also the case that $\mu (\Omega)=0$. Two signed measures $\mu_1$ and $\mu_2$ are mutually singular, denoted $\mu_1 \perp \mu_2$, if they are concentrated on disjoint sets. 

Lebesgue's decomposition theorem states that a $\sigma$-finite measure $\mu$ can  be the uniquely decomposed into its absolutely continuous part and singular part with respect to Lebesgue measure: $\mu =\mu_{ac} + \mu_{s}$, where $\mu_{ac} \ll H^m$ and  $\mu_{s} \perp H^m$; furthermore $\mu_{ac} = hdH^m$ for a unique $h\in L^1(H^m)$ called the Radon-Nikodym derivative of $\mu_{ac}$ (with respect to $H^m$). Note that $\mu_{s} \perp \mu_{ac}$.

The Hahn-Jordan decomposition theorem states that any signed measure $\mu$ can be uniquely decomposed into its positive and negative parts: $\mu =\mu_+ - \mu_-$, where $\mu_+, \mu_-$ are positive measures and $\mu_{+} \perp \mu_{-}$. The variation of $\mu$ is denoted by $|\mu | :=\mu_+ + \mu_-$.

%\begin{lemma}
It is a standard fact\label{measure_decompoition} that for  a $\sigma$-finite measure $\mu$ the Hahn-Jordan decomposition and the Lebesgue decomposition commute: $\mu_+= [\mu_{ac}]_+ + [\mu_{s}]_+$ and $\mu_-= [\mu_{ac}]_- + [\mu_{s}]_-$.
%\end{lemma}
%\begin{proof}
%It follows from $\mu_{ac} \perp \mu_{s}$ that $[\mu_{ac}]_{\pm} \perp [\mu_{s}]_{\pm}$. Decomposing $\mu_{ac} =[\mu_{ac}]_+ - [\mu_{ac}]_-$  and $\mu_{s} =[\mu_{s}]_+ - [\mu_{s}]_-$ we note $[\mu_{ac}]_+ \perp [\mu_{ac}]_-$ and $[\mu_{s}]_+ \perp [\mu_{s}]_-$. Hence $([\mu_{ac}]_+ + [\mu_{s}]_+) \perp ([\mu_{ac}]_- + [\mu_{s}]_-)$. Also $\mu = \mu_{ac} + \mu_{s}= ([\mu_{ac}]_+ + [\mu_{s}]_+) - ([\mu_{ac}]_- + [\mu_{s}]_-)$. Since $[\mu_{ac}]_+ + [\mu_{s}]_+$ and $[\mu_{ac}]_- + [\mu_{s}]_-$ are positive and mutually singular, it follows from uniqueness of the Hahn-Jordan decomposition that these are the positive and negative parts of $\mu_{ac} + \mu_{s}$.
%\end{proof}
 
 \begin{theorem}[Existence of optimal potentials]\label{T:existence of optimizers}
Let $f,g$ and $\bar h$ be continuous and strictly positive on the compact, unit-volume sets 
$X \subset \R^m,Y\subset \R^n$ and $X \times Y$ respectively. % of unit volume. 
Fix $\eta>1$ and $s \in L^1(X \times Y)$. If $\Gamma^{\bar h/\eta}(f, g)$ is not empty,
%If there exists a probability density $h \le \bar h/\eta$ having $f$ and $g$ as its marginals,
then there exist functions $(u,v) \in L^1(X) \oplus L^1(Y)$ such that $I_*=I(u,v)$ % = \tilde I^{(s,\bar h)}(U,V; f,g)$
in \eqref{convex dual}.
\end{theorem}

\begin{proof}
Let $(u_i,v_i)$ be a minimizing sequence for \eqref{convex dual}: $u_i \in L^1(X)$ and $v_i \in L^1(Y)$ such that % $(f_i,g_i) \to (f,g)$ uniformly and 
$I(u_i,v_i) \to I_*$.  Fix $\epsilon \le \min\{f(x)^{\pm 1},g(y)^{\pm 1}\}$. 
%are bounded away from zero and infinity,
%for $i$ sufficiently large similar bounds hold for $f_i$ and $g_i$, which we may also normalize to be probability
%densities without loss of generality.  
Since $\bar h>0$ is bounded away from zero, taking $\epsilon>0$ smaller if
necessary ensures $\bar h(x,y) \ge \epsilon f(x) g(y)$ throughout $X \times Y$.  
%Taking $\epsilon$ smaller if necessary ensures $\min\{f_i(x)^{\pm 1},g_i(y)^{\pm 1}\} \ge \epsilon$.
%Reducing $\eta>1$ if necessary, NEED TO USE SOLUTION OF FRECHET PROBLEM 
%\cite{Levin84} TO SHOW THERE EXISTS $0 \le h_i \le \bar h/\eta$ having $f_i$ and $g_i$ as marginals.
Adding a constant to $u_i$ and subtracting the same constant from $v_i$ ensures $\mean{u_if} = \mean{v_ig}$.
Now Proposition \ref{P:coercivity} asserts a bound for $\|u_i\|_{L^1} + \|v_i\|_{L^1}$ independent of $i$.
Letting $U_i$ and $V_i$ denote the measures with Lebesgue densities $u_i$ and $v_i$ and
recalling $M(X) = C(X)^*$, Alaoglu's theorem provides a weak-$*$
convergent subsequence also denoted $(U_i,V_i)$ with limit $(U_i,V_i) \to (U,V) \in M(X)\oplus M(Y)$.
Since $\tilde I(U_i,V_i) = I(u_i,v_i) \to I_*$
by construction, Lemma \ref{L:lower semicontinuity} implies %the desired conclusion
$\tilde I(U,V)\le I_*$.  
%On the other hand $\tilde I_*$ is an infimum by definition \eqref{relaxed dual}, so equality must hold and $(U,V)$ are the desired minimizers.
We will next argue that $\tilde I(U,V)\ge I_*$.

Let $S = s H^n\otimes H^m $. Since the Hahn-Jordan and the Lebesgue decompositions of a measure commute, as mentioned above, 
\begin{eqnarray}\label{decompositions_commute}
\nonumber \lefteqn{\left[S + U\otimes H^n + H^m\otimes V\right]_+}\\
 &=&\left[ S + U_{ac}\otimes H^n + H^m\otimes V_{ac} + U_{s}\otimes H^n + H^m\otimes V_{s}\right]_+\\
\nonumber &=& \left[ S + U_{ac}\otimes H^n + H^m\otimes V_{ac} \right]_+ + \left[   U_{s}\otimes H^n +H^m\otimes  V_{s}\right]_+.
\end{eqnarray}
Therefore
\begin{eqnarray*}
\nonumber \lefteqn{\tilde I(U,V)}\\
\nonumber  &=& \tilde I(U_{ac},V_{ac}) + \iint \bar h\, d\left[  U_{s}\otimes H^n +H^m \otimes V_{s}\right]_+ - \int f\, dU_{s} - \int g\, dV_{s}\\
 &\ge&  I(u,v) + \iint h\, d\left[  U_{s}\otimes H^n + H^m \otimes V_{s}\right] - \int f\, dU_{s} - \int g\, dV_{s}\\
\nonumber &=& I(u,v),\\
\nonumber &\ge& I_*
\end{eqnarray*}
where $u$ and $v$ are the Radon-Nikodym derivatives of $U_{ac}$ and $V_{ac}$ respectively: $u\, dx = dU_{ac}$ and $v\, dy= dV_{ac}$.
Thus $\tilde I(U,V) \ge I(u,v) \ge I_* \ge \tilde I(U,V)$, showing $(u,v)$
is the desired minimizer.
\end{proof}

%\begin{corollary}[Characterizing optimality]\label{C:Characterizing optimality}
%Under the hypotheses of Theorem \ref{T:existence of optimizers}, 
%the infimum \eqref{convex dual} is attained
%at some point $ I(u,v) = I_*$ and conditions \eqref{complementary slackness}
%become necessary and sufficient to characterize those $h$ which maximize \eqref{primal}.
%\end{corollary}

%\begin{proof}
%Follows from Theorem \ref{T:existence of optimizers} and Corollary \ref{C:Levins_duality}.
%\end{proof}

\section{From Levin's to Kantorovich's duality} \label{S5}

In this section, we briefly discuss how to derive the classical Kantorovich duality of optimal transport from Levin's duality using some of the insights of previous sections. For simplicity, assume that $f$ and $g$ are
strictly positive, continuous probability densities on the compact 
sets  $X\subset \R^m$ and $Y\subset \R^n$,
where both $X$ and $Y$ have unit volume. 
Let $\Gamma(f,g)$ denote the set of all joint probability measures with $f$ and $g$ as marginals. 
Fix $s \in C(X \times Y)$ continuous and define
\[
%M_S:=\left\{(U,V)\in M(X)\oplus M(Y)|\: [S + U\otimes H^n+ H^m\otimes V]_+=0\right\}.
L_s:=\left\{(u,v)\in L^1(X)\oplus L^1(Y)|\: s(x,y) + u(x)+ v(y) \leq 0 \right\}.
\]
We will recover the expression of Kantorovich duality as
%duality of a similar flavour to the stronger result 
stated in \cite[Theorem 5.10]{Villani09}.
%Here elements of $M_S$ need not be absolutely continuous.  As partial compensation for this relaxation,
%we obtain that the dual infimum \eqref{4} is attained for costs $c=-s \in L^1(\R^{m+n})$ which,
%unlike those of \cite{Villani09},  are not necessarily lower semicontinuous. O
We remark that although the elements of $L_s$ need not be continuous, for Lipschitz continuous costs 
it is well-known that the infimum on the right hand side of \eqref{4} below is attained by Lipschitz continuous densities $(u,v)$.

\begin{theorem}[Duality for the unconstrained problem]\label{Kantorovich}
Let $X\subset \R^m$ and $Y\subset \R^n$ be compact unit volume sets equipped with positive continuous 
probability densities $f \in C(X)$ and $g \in C(Y)$. 
%Let $L_s:=\left\{(u,v)\in L^1(X)\oplus L^1(Y)|\: s(x,y) + u(x)+ v(y) \leq 0 \right\}$. 
If 
%$S \in M(X \times Y)$ is given by a continuous density 
$s \in C(X \times Y)$ then
\begin{equation}\label{4}
\sup_{H\in \Gamma(f,g)} \iint s(x,y) d H(x,y) = \inf_{(u,v)\in L_s} \left(- \int f(x) u(x)dx - \int g(y) v(y)dy\right),
\end{equation}
%If $s \in L^1(X \times Y)$ is merely integrable, still 
and both infimum and supremum \eqref{4} are attained.
\end{theorem}

\begin{proof}
Let $I^*(\infty)$ denote the left hand side of \eqref{4} and $I^*(\infty)$ denote the right hand side of \eqref{4}. 
Note that by a density argument, it is enough to let the $\sup$ in $I^*(\infty)$ range over absolutely continuous measures $d H = hdxdy\in\Gamma(f,g)$.

The inequality $I^*(\infty) \leq I_*(\infty)$ follows easily from the following inequality which holds for all $(u,v) \in L_s$ and any density $h$ such that $ hdxdy\in\Gamma(f,g)$:

\begin{eqnarray*}
\iint s h dxdy
%&=& \iint  h \,dS\\
& = &\iint  h(s+u+v)dxdy - \int f udx - \int g vdy\\
& \leq & - \int f u dx- \int g vdy.
\end{eqnarray*}
The last inequality follows from the definition of $L_s$. Taking $\sup$ on the left hand side and $\inf$ on the right hand side gives the desired inequality.
%\[
%\sup_{ H\in \Gamma(f,g)}\iint s(x,y)d H(x,y)\le\inf_{( U, V)\in M_S}\left(- \int fd U - \int gd V\right).
%\]

Observe that to prove the other inequality, $I^*(\infty) \geq I_*(\infty)$, it is enough to prove existence of $H\in\Gamma(f,g)$ and $(u,v)\in L_s$ satisfying
\[
\iint s\,d H\ge- \int f udx - \int g vdy.
\]

Let $h(x,y):=f(x)g(y)$ and $K:=\max_{X \times Y}h(x,y)+1$, and note that $hdxdy\in\Gamma(f,g)$. Fix $\eta >1$ such that $K \geq \eta\, \max_{X \times Y}h(x,y)$, so that $k \geq \eta h$ for all $k\geq K$. Note also that $\Gamma^{\bar h_k}(f,g)\not=\emptyset$ if $\bar h_k=k1_{X\times Y}$.

Defining $I_k(u,v)$, $I_*(k)$, and $I^*(k)$ similarly to $I(u,v)$, $I_*$, and $I^*$ but with $\bar h$ replaced by $\bar h_k$, we deduce from Levin's duality $I^*(k)=I_*(k)$ the existence of functions $h_k\in \Gamma^{\bar h_k}(f,g)$ and $(u_k,v_k)\in L^1(X)\oplus L^1(Y)$ optimizing $I^*(k)$ and $I_*(k)$ respectively. Levin's duality reads
\begin{equation}\label{1}
\iint h_k\,s dxdy=I_k(u_k,v_k)
\end{equation}
for any $k\ge K$. Since $I_k(u+\alpha, v-\alpha)=I_k(u,v)$ for $\alpha\in\R$, there is no loss of generality to assume that
\[
\mean{u_k f}=\mean{v_k g}.
\]

\begin{claim}
The sequence $\{h_k\}_{k\ge K}$ is weak-$*$ precompact 
in the space of probability measures and every limit point $H$ satisfies $H\in \Gamma(f,g)$.
\end{claim}

\begin{proof}
Since $\{ h_k\}_{k\ge K}$ is a sequence of probability densities on a compact domain, the sequence is 
weak-$*$ precompact in $M(X\times Y)$. Therefore, there exists a probability measure $H$ and a subsequence (not relabeled) such that $h_k\, dxdy \to H$ in the sense of weak-$*$ convergence. Thus, for every continuous function $\zeta=\zeta(x)$, it holds that
\[
\int \zeta fdx\;=\;\lim_{k\to\infty} \iint \zeta(x) h_k(x,y) dxdy  \;=\;  \iint \zeta(x) d H(x,y) ,
\]
and the analogous identity holds for the $y$-marginals. Hence $H\in\Gamma(f,g)$.
\end{proof}

In particular, since $s$ is continuous, upon extracting a subsequence, passing to the limit $k\to \infty$ in \eqref{1} yields that
\begin{equation}\label{3}
\iint s\,dH= \lim_{k\to \infty} I_k(u_k,v_k).
\end{equation}

We next observe that the $L^1$-bound on the sequences $\{u_k\}_{k\ge K}$ and $\{v_k\}_{k\ge K}$ obtained from Lemmas \ref{L:mean bound} and \ref{L:oscillation bound} is independent of $k\ge K$. In this context, we state for further references
\begin{equation}\label{B: bound}
\mean{u_kf} + \mean{v_kg} \leq \frac{I^*(\infty)+\eta||s||_{L^1} ||h||_\infty}{\eta -1},
\end{equation}
as a consequence Lemma \ref{L:mean bound}.

Choosing $\epsilon >0$ so that $K \geq \epsilon f(x)g(y)$ and $\min \{f(x),g(y)\} \geq \epsilon$, Proposition \ref{P:coercivity} (see Remark \ref{R:uniform bound}) implies that $||u_k||_{L^1}$, and $||v_k||_{L^1}$ are controlled by a bound which depends only on $\epsilon, \eta, I^*(\infty),||s||_{L^1}$, and $||h||_\infty$, all of which are independent of $k$.

%Indeed, in view of \eqref{1b} and choosing $\bar h=\bar h_k$, $\eta=k$, and $h=1_{X\times Y}$, the statement in Lemma %\ref{L:mean bound} implies that
%\begin{equation}\label{3a}
%2|\mean{u_k f}|\le I_k(u_k,v_k) + \frac{k}{k-1}\|1_{X\times Y} s\|_1\stackrel{\eqref{3}}{\le}C_1
%\end{equation}
%for some constant $C_1>0$ uniformly in $k$. Likewise, since $\bar h_k\ge fg$ for $k\ge K$, the oscillation bound in %Lemma \ref{L:oscillation bound} becomes
%\[
%\frac16\|u_k - \mean{u_k f}\|_1\le I_k(u_k,v_k) + \|sfg\|_1 + 2|\mean{u_k f}|\stackrel{\eqref{3}}{\le} C_2 + 2|\mean{u_k f}|,
%\]
%for another constant $C_2>0$ uniformly in $k$. Combining these two estimates, we obtain
%\[
%\|u_k\|_1\le \frac1{\inf|f|}\|u_k f\|_1\le C,
%\]
%for some $C>0$ uniformly in  $k$, where we have used the assumption that $f$ is strictly positive. 

Arguing as in the proof of Theorem \ref{T:existence of optimizers}, we may thus extract a weak-$*$ convergent subsequence (not relabeled) $(u_k,v_k)\to (U,V)\in M(X) \oplus M(Y)$ and obtain
\[
\lim_{k\to \infty} I_k(u_k,v_k)\ge \lim_{k\to \infty} \left(-\int u_k fdx - \int v_k gdy\right) = - \int fdU -\int gdV,
\]
by the continuity of $f$ and $g$.
In particular, we deduce from \eqref{1} and \eqref{3} that
\begin{equation}\label{2}
\iint s\, dH\ge - \int fdU -\int gdV.
\end{equation}
Moreover, the definition of $\bar h_k$ and Lemma \ref{L:lower semicontinuity} imply that
\begin{eqnarray*}
\lefteqn{\iint d[S + U\otimes H^n+ H^m\otimes V]_+ }\\
&\le&\liminf_{k\to \infty} \iint[s(x,y) + u_k(x)+v_k(y)]_+dxdy\\
&=& \liminf_{k\to \infty} \frac1k\left(I_k(u_k, v_k) + \mean{u_k f} + \mean{v_k g}\right).
\end{eqnarray*}
Since the limit on the right hand side is zero, a consequence of  \eqref{3} and \eqref{B: bound}, it follows that $ [S + U\otimes H^n+ H^m\otimes V]_+=0$. Now \eqref{decompositions_commute} implies that  $\left[ S + U_{ac}\otimes H^n + H^m\otimes V_{ac} \right]_+$ and $ \left[   U_{s}\otimes H^n +H^m\otimes  V_{s}\right]_+$ are both zero. In particular, 
$\left[ U_{s} \right]_+$ and $\left[ V_{s} \right]_+$ are bother zero, and
with $u := dU_{ac}/dx$ and $v := dV_{ac}/dy$ denoting Radon-Nikodym derivatives, we have $[s+u+v]_+=0$, or equivalently $s+u+v\in L_s$.
It follows that
\begin{eqnarray*}
- \int fdU -\int gdV
&=& - \int fudx -\int gvdy - \int fdU_s -\int gdV_s \\
&\geq & - \int fudx -\int gvdy.
%&=& \iint h\, d\left[ S + U\otimes H^n +H^m \otimes V \right]_+ - \int f\, dU - \int g\, dV\\
%&=& \iint h\, d\left[ S + U_{ac}\otimes H^n +H^m \otimes V_{ac}\right]_+ - \int f\, dU_{ac} - \int g\, dV_{ac}\\
%&+&\iint h\, d\left[  U_{s}\otimes H^n +H^m \otimes V_{s}\right]_+ - \int f\, dU_{s} - \int g\, dV_{s}\\
%&=&  - \int f u - \int g v
\end{eqnarray*}

Hence \eqref{2} becomes $\iint s\, dH\ge - \int f udx -\int g vdy$ which implies the desired inequality $I^*(\infty) \geq I_*(\infty)$. We have thus established \eqref{4}
and existence of optimizers.
\end{proof}

\section{Perspectives for future work}\label{S6}

Of considerable interest to us is the question of showing some regularity for
the minimizing potentials $(u,v)$ --- perhaps under stronger restrictions on $(f,g,\bar h)$ and $s$.  
For example,  are they continuous or differentiable; might  $u$ and $v$ belong to some H\"older or Sobolev space or
--- as in the unconstrained version $\bar h=+\infty$ of the problem \cite{McCannGuillen13} \cite{Villani09} 
--- inherit Lipschitz and semiconcavity properties from the cost $c=-s$?
In view of the characterization \eqref{complementary slackness} for 
$h \in \Gamma^{\bar h}(f,g)$ to maximize the expected value of $s$,  this is closely
related to smoothness for the free boundary of the set $W$ such that the optimizer $h= \bar h 1_W$.
%Where $0\ne \det[ \partial^2 s/\partial x^i\partial y^j]$ 
If $m=n$,  this set is known to be unique (up to sets of $H^{2n}$ measure zero) near points 
where $0\ne \det [\partial^2 s/\partial x^i\partial y^j]$ \cite{KormanMcCann12p}.  Simple examples 
%from the same sources 
show the boundary of $W$ can have isolated singularities \cite{KormanMcCann12p} \cite{KormanMcCann13},  
but is it a smooth hypersurface otherwise?
Does $W$ even have finite perimeter?  Where the derivative of $s(x,y)+u(x)+v(y)$ is non-vanishing,
such questions are related to smoothness of $u$ and $v$ by the implicit function theorem.

%\section{Appendix: Decompositions of measures}

\end{document}